\documentclass[a4paper, 11pt]{amsart}
\usepackage[margin=2.5cm]{geometry}
\usepackage[utf8]{inputenc}
\usepackage[english]{babel}
\usepackage{amssymb}
\usepackage{amsmath}
\usepackage{amsthm}
\usepackage{amsfonts}
\usepackage{array}
\usepackage{comment,stmaryrd}
\usepackage{enumitem}
\usepackage{hyperref}
\usepackage{mathrsfs}
\usepackage{tikz}
\usetikzlibrary{arrows,decorations.markings}
\usepackage{tikz-cd, mathtools}
\usepackage{mathabx}
\usepackage{xfrac}
\usepackage{graphicx} 

\theoremstyle{plain}
\newtheorem{theorem}{Theorem}[section]

\newtheorem{proposition}[theorem]{Proposition}

\theoremstyle{definition}

\newtheorem{example}[theorem]{Example}
\newtheorem{remark}[theorem]{Remark}
\newtheorem{challenge}[theorem]{Challenge}

\DeclareMathOperator{\Spec}{Spec}

\DeclareMathOperator{\Rep}{Rep}

\DeclareMathOperator{\A}{\mathbb{A}}
\DeclareMathOperator{\cP}{\mathcal{P}}

\makeatletter
\let\@wraptoccontribs\wraptoccontribs
\makeatother

\author{Nsibiet E. Udo}
\address{University of Calabar, Nigeria}
\email{nsibietudo@unical.edu.ng}

\author{Praise Adeyemo}
\address{University of Ibadan, Nigeria}
\email{ph.adeyemo@ui.edu.ng}

\author{Bal\'{a}zs Szendr\H{o}i}
\address{University of Vienna, Austria}
\email{balazs.szendroi@univie.ac.at}

\contrib[with an Appendix by]{Stavros Argyrios Papadakis}
\address{University of Ioannina, Greece}
\email{spapadak@uoi.gr}

\title[Ideals, representations and a symmetrisedg Bernoulli triangle]{Ideals, representations and \\ a symmetrised Bernoulli triangle}
\date{July 2024}

\begin{document}

\begin{abstract} We study some representations of symmetric groups arising from a certain ideal in the coordinate ring of affine $n$-space. Our results give graded and representation-theoretic enhancements of the sequence of numbers $c_n=1 + (n-2)\cdot 2^{n-1}$, sequence 337 of the Online Encyclopaedia of Integer Sequences, involving a symmetric version of the Bernoulli triangle.  
\end{abstract}

\maketitle

\section*{Introduction}

In this paper, we study some natural representations of symmetric groups arising from the binomial ideal 
\[ I_n = \langle x_2x_3 \cdots x_n - x_1, x_1x_3 \cdots x_n - x_2, \ldots, x_1x_2 \cdots x_{n-1} - x_n \rangle \lhd R_n \] 
of the free polynomial algebra $R_n=F[x_1,\ldots, x_n]$ over a field $F$, assumed algebraically closed of characteristic zero. The ideal $I_n$ is inhomogeneous, $S_n$-invariant, radical, and of finite codimension 
$c_n=1 + (n-2)\cdot 2^{n-1}$ (Proposition~\ref{propIn}), a quantity that arises in various corners of combinatorics and combinatorial topology~\cite{OEIS_337}. Our first result, Theorem~\ref{thm1}, computes the $S_n$-character of the quotient $R_n/I_n$, giving a representation-theoretic refinement of the formula defining~$c_n$. We then consider a certain monomial specialisation $J_n\lhd R_n$ of $I_n$. Our second result, Theorem~\ref{thm2}, computes the Hilbert series of the corresponding Artinian quotient $R_n/J_n$ in terms of partial sums of binomial coefficients. The coefficients of the Hilbert series can be seen as one row of what we call here the symmetrised Bernoulli triangle. As a biproduct of the proof, we discover that the ideal $J_n$ is, somewhat surprisingly, invariant under a different symmetric group $S_{n-1}$. Our next result, Theorem~\ref{thm3}, gives another representation-theoretic refinement of the quantity $c_n$, now also involving the grading. We finally turn to an intermediate homogeneous specialisation $K_n$ of $I_n$. We show in Theorem~\ref{thmG} that the corresponding graded Artinian quotient $R_n/K_n$, as opposed to $R_n/J_n$, is Gorenstein, explaining the symmetry of the Hilbert function in Theorem~\ref{thm2}. We conclude with a question about the graded character of this quotient. 

The combinatorics of the symmetrised Bernoulli triangle is discussed in Section~\ref{sec1}.
The inhomogeneous ideal $I_n$ is discussed in Section~\ref{sec2}, the monomial specialisation $J_n$ is the subject of Section~\ref{sec3}, while Section~\ref{sec4} studies the homogeneous ideal $K_n$. The proof of the key Gorenstein property is the subject of the Appendix.

\vspace{0.1in}

\noindent{\bf Conventions } We work over an algebraically closed field $F$ of characteristic $0$. For $n\geq 1$, we denote the ring of finite-dimensional representations of the symmetric group $S_n$ over~$F$ by $\Rep(S_n)$, with unit $[1]\in \Rep(S_n)$. For $V$ a finite-dimensional $S_n$-representation, we denote by 
$\chi(V)\in \Rep(S_n)$ the corresponding ring element. If $V=\oplus_r V_r$ is a graded finite-dimensional $S_n$-representation, then let 
$\chi_t(V) = \sum_r \chi(V_r)t^r\in \Rep(S_n)[t]$.
We denote by $\cP(n)$ the set of all subsets of $\{1,\ldots,n\}$ and $[\cP(n)]\in \Rep(S_n)$ the corresponding finite-dimensional permutation representation, of dimension $2^n$. Further, for $0\leq k\leq n$, $\cP_k(n)$ will denote the set of all $k$-element subsets, and $[\cP_k(n)]\in \Rep(S_n)$ the corresponding finite-dimensional permutation representation, of dimension $n \choose k$. If $n$ is odd, then the set of all even-element subsets of 
$\{1,\ldots,n\}$ gives an isomorphic representation to the set of all odd-element subsets; we denote by $\frac{1}{2}[\cP(n)]\in \Rep(S_n)$ this representation, of dimension $2^{n-1}$. 

\vspace{0.1in}

\noindent{\bf Acknowledgements } The authors would like to thank Livia Campo, Miles Reid, and especially Joachim Jelisiejew for comments, and Stavros Argyrios Papadakis for contributing the Appendix.  N.E.U. would also like to thank Dominic Joyce for his mentorship on a related project. 

\section{A symmetrised Bernoulli triangle}\label{sec1}

For $n\geq 0$ and $0\leq k\leq n$ integers, partial sums of binomial coefficients 
\[ b_{n,k} = \sum_{j=0}^{k} {n \choose j}
\]
form a triangular array of positive integers sometimes called the {\em Bernoulli triangle}. The edge terms are $b_{n,0}=1$ and $b_{n,n}=2^n$, but otherwise this triangle can also be defined by the basic recursion
\begin{equation}\label{recursion}
     b_{n,k} = b_{n-1,k-1} + b_{n-1,k}
\end{equation}
defining the Pascal triangle. 

We will be interested in a symmetrized, re-indexed version of the Bernoulli triangle. Define for $n\geq 2$ and $0\leq k \leq 2n-4$ the following set of positive integers: 
\[ a_{n,k} = \left\{ \begin{array}{ll} b_{n-1,k}  & \mbox{ if } 0\leq k \leq n-2, \\  b_{n-1,2n-4-k}  & \mbox{ if } n-2\leq k \leq 2n-4.\end{array} \right.
\]
We have in particular $a_{n,0} =a_{n,2n-4} = 1$, and the sequence $\{a_{n,k} : 0\leq k \leq 2n-4\}$ is symmetric and strictly increasing until its middle term $a_{n,n-2} = b_{n-1,n-2} = 2^{n-1}-1$. Here are the rows of our triangle for $2\leq n \leq 6$: 

\[\begin{array}{ccccccccc} &&&& 1\\ &&& 1 & 3 & 1\\ && 1 & 4 & 7 & 4 & 1  \\ & 1 & 5 & 11 & 15 & 11 & 5  &1  \\  1 & 6 & 16 & 26 & 31 &  26 & 16 & 6  &1\end{array}\] 

\vspace{.1in}

The row sums $\{1,5,17,49,129, \ldots\}$ of our triangle turn out to be numbers arising elsewhere~\cite{OEIS_337}. 
\begin{proposition} We have
\[\sum_{k=0}^{2n-4}a_{n,k} = 1 + (n-2)\cdot 2^{n-1}. \]
\end{proposition}
\begin{proof} This follows by induction using the recursion~\eqref{recursion}.
\end{proof}

For further reference, we re-write this last formula also in the form
\begin{equation} \label{identity}
\sum_{j=0}^{n-2} (2j+1) {n-1 \choose n-2-j} =1 + (n-2)\cdot 2^{n-1}. 
\end{equation}

\section{The inhomogeneous ideal}\label{sec2}

For $n\geq 3$, consider the inhomogeneous, binomial~\cite{binom} ideal
\[ I_n = \langle x_2x_3 \cdots x_n - x_1, x_1x_3 \cdots x_n - x_2, \ldots, x_1x_2 \cdots x_{n-1} - x_n \rangle \lhd R_n. \] 
The $n=2$ case is degenerate and not particularly exciting; it will fit with all our results if we set 
\[ I_2 = \langle x_1, x_2\rangle \lhd F[x_1, x_2]. \]
For $n=3$, the ideal 
\[I_3 = \langle x_2x_3-x_1, x_1x_3-x_2, x_1x_2-x_3\rangle \lhd F[x_1, x_2, x_3]\] 
appeared in~\cite[Sect.3.1]{scattering}, which inspired our construction. 
\begin{proposition}\label{propIn} For all $n\geq 2$, the ideal $I_n \lhd  R_n$ is a reduced complete intersection ideal, of finite codimension $c_n=1 + (n-2)\cdot 2^{n-1}$. 
\begin{proof} For $n=2$, the statement is clear. For $n\geq 3$, let us determine \[X_n = \Spec R_n/I_n\subset\A^n\] as a set first. 
Treating the ideal generators as equations for $X_n$, we deduce $x_i^2 = p$, where we have set $p=x_1\cdots x_n$, and then $p^n = p^2$. We get either $p=0$, which then implies $x_i=0$, or else $p$ is one of the $(n-2)$-nd roots of unity in the field~$F$. For each different choice of root of unity, we can choose independent square roots for the first $(n-1)$ variables, and then the last one is determined. Hence indeed $\dim X_n=0$, and $X_n$ consists of $1+(n-2)\cdot 2^{n-1}$ points. 

It remains to show that $X_n$ is reduced at each of these points. But the ideal $I_n$ contains 
$x_i - \prod_{j\neq i} x_j$, and so also
\[ x_i - \prod_{j\neq i}\prod_{k\neq j} x_k = x_i(1-p^{n-2}). 
\]
Hence in the primary decomposition of $I_n$, the component at the origin contains each $x_i$, so $X_n$ is reduced at the origin. The argument for the other points is similar. 
\end{proof}
\end{proposition}

Consider the standard action of the symmetric group $S_n$ on the algebra $R_n$. The ideal~$I_n$ is clearly invariant under this action, so the Artinian $F$-algebra $R_n/I_n$ has the structure of an $S_n$-module. The following is the main result of this section, giving a representation-theoretic enhancement of the set of numbers $\{ 1 + (n-2)\cdot 2^{n-1} : n\geq 2\}$. 
\begin{theorem}\label{thm1} For $n\geq 2$, we have 
\[ \chi\left(R_n/I_n\right) =  [1] + \frac{n-2}{2} [\cP(n)] \in \Rep(S_n), 
\]
where for $n$ odd, the meaning of $\frac{1}{2}[\cP(n)] \in \Rep(S_n)$ was introduced in the Introduction. 
\end{theorem}
\begin{proof} Since the vanishing locus $X_n$ of~$I_n$ is reduced, the representation $R_n/I_n$ is just the 
permutation representation of $S_n$ on the point set $X_n$. The (reduced) origin in $X_n$ corresponds to the trivial representation. To determine the other orbits, fix a primitive $2(n-2)$-nd root of unity
$\xi\in K$. For a point $(x_1,\ldots, x_n) \in X_n$, the possible values of the product $p=\prod_j x_j$ are $p=\xi^{2k}$ for $0\leq k \leq n-3$. For a fixed value $p=\xi^{2k}$, we get $x_j= {\sigma_j}\xi^k$, where $\sigma_j\in\{\pm 1\}$, with the condition 
\[ p=\xi^{2k} = \prod_j x_j = \left(\prod_j \sigma_j\right) \xi^{nk}.
\]
So we get 
\[
\prod_j \sigma_j = \xi^{(n-2)k}.
\]
For even $k$, this product of signs is $+1$, whereas for odd $k$, it is $-1$. 

Suppose first that $n$ itself is odd. Then for each value $0\leq k \leq n-3$, the variables with negative sign $\sigma_j=-1$ correspond to a choice of some subset of $\{1,\ldots, n\}$ of size the same parity as $k$. For fixed $k$ and fixed size, the corresponding points form one orbit of $S_n$. So for each $0\leq k \leq (n-3)$, we get exactly the representation $\frac{1}{2}[\cP(n)]\in \Rep(S_n)$: subsets of $\{1,\ldots, n\}$ of the same parity as $k$. Since there are $(n-2)$ possible values of $k$, we indeed get $(n-2)\cdot \frac{1}{2} [\cP(n)]$ as the contribution from the orbits away from the origin. 

If $n$ is even, it is still true that for fixed $k$, we get subsets of $\{1,\ldots, n\}$ of the same parity as $k$ as orbits. In this case, for even $k$, we get all even element subsets, whereas for odd $k$ we have all odd element subsets. But there are exactly $\frac{n-2}{2}$ even numbers $k$ with $0\leq k \leq n-3$, and the same number of odd numbers. So we simply get $\frac{n-2}{2}$ times the permutation representation on all subsets of $\{1,\ldots, n\}$, hence the result. 
\end{proof}

\begin{example} Let $n=3$, so that our ideal is
\[ I_3 = \langle x_2x_3- x_1, x_1x_3  - x_2, x_1x_2  - x_3 \rangle \lhd F[x_1, x_2, x_3]. \] 
Then we get 
\[ X_3 = \{ (0,0,0), (1,1,1), (-1,-1,1), (-1,1,-1), (1,-1,-1)\}\subset\A^3. 
\]
There are two trivial $S_3$-orbits and one orbit of size $3$. Indeed, in this case
\[[1]+\frac{1}{2}[\cP(3)] = 2\cdot [1]+[\cP_1(3)]\in \Rep(S_3).\]
\end{example}

\begin{remark} In~\cite{OEIS_337} and references listed there, several seemingly different interpretations are given for the sequence of positive integers \[\{ 1 + (n-2)\cdot 2^{n-1}: n\geq 2\} = \{1,5,17,49,129,321,\ldots\}.\] Some of these interpretations are combinatorial, counting the number of elements in certain sets. But it seems to us that none of the sets listed in~\cite{OEIS_337} carries a natural representation of the symmetric group of the appropriate size. Indeed, we have not found representation-theoretic enhancements of this sequence anywhere in the literature. 
\end{remark}

\section{A monomial specialisation}\label{sec3}

Equip the polynomial algebra $R_n$ with the graded reverse lexicographic order, the standard order GRevLex chosen by Macaulay2~\cite{m2}, which we used to discover our results. Let
\[ J_n = \mathop{\mathrm{in}_{\mathrm{GRevLex}}} I_n\lhd R_n
\]
be the corresponding initial ideal, a monomial ideal of $R_n$. As $J_n$ is a flat specialisation of $I_n$, $R_n/J_n$ is a graded Artinian $F$-algebra of finite $F$-dimension $c_n=1 + (n-2)\cdot 2^{n-1}$. 
Consider its Hilbert series
\[h_n(t) = \sum_{k=0}^\infty t^k \dim_F(R_n/J_n)_k .
\]
As $R_n/J_n$ is Artinian, this is in fact a finite sum. 

\begin{theorem} \label{thm2} For $n\geq 2$, the Hilbert series of the graded Artinian $F$-algebra $R_n/J_n$ is 
\[h_n(t) = \sum_{k=0}^{2n-4} t^k a_{n,k}.\]
In other words, the coefficients are given precisely by the corresponding row of the symmetrised Bernoulli triangle. 
\end{theorem}

To motivate the proof below, we first discuss the simplest non-trivial example. 

\begin{example} Continue with $n=3$, so that our inhomogeneous ideal, of colength $c_3=5$, is
\[ I_3 = \langle x_2x_3- x_1, x_1x_3  - x_2, x_1x_2  - x_3 \rangle \lhd F[x_1, x_2, x_3]. \] 
This ideal contains, alongside its generators, the elements
\begin{eqnarray*} 
x_1^2-x_3^2 &  = & x_1(-x_2x_3+x_1)+ x_3(x_1x_2-x_3), \\
x_2^2-x_3^2 &  = & x_2(-x_1x_3+x_2)+ x_3(x_1x_2-x_3), \\
x_3^3-x_3  & = & (-x_3^2+1) \cdot(x_1x_2-x_3) + x_1\cdot(x_1x_3-x_2) + x_1x_3\cdot(x_2x_3-x_1).
\end{eqnarray*} 
The initial ideal $J_3$ therefore contains the monomials
\[x_1^2, x_2^2, x_1x_2, x_1x_3, x_2x_3, x_3^3.\]
It can be checked that these elements generate an ideal of colength $5$, so we must in fact have
\[J_3=\langle x_1^2, x_2^2, x_1x_2, x_1x_3, x_2x_3, x_3^3\rangle \lhd F[x_1, x_2, x_3].\]
A homogeneous basis of $F[x_1, x_2, x_3]/J_3$ is given by $\{1, x_1, x_2, x_3, x_3^2\}$, and hence the Hilbert series of this graded $F$-algebra is $1+3t+t^2$ indeed. 
\end{example}

As in the example, we start the proof of Theorem~\ref{thm2} by exhibiting a large set of monomials in $J_n$. 
\begin{proposition} \label{prop} \ 

\begin{enumerate}
\item[{\rm (i)}] The following monomials form a generating set of the ideal $J_n$:
\begin{enumerate}
    \item[{\rm (a)}] the squares $x_i^2$ for $1\leq i \leq n-1$;
    \item[{\rm (b)}] the product $x_1x_2\ldots x_{n-1}$;
    \item[{\rm (c)}] for any $0\leq j \leq n-2$ and any subset $T\subset \{1,\ldots, n-1\}$ of size $n-2-j$, the monomial \[m(T)=x_n^{2j+1}\prod_{i\in T} x_i.\]
\end{enumerate}
\item[{\rm (ii)}] The monomials 
 \[m(T,s) = x_n^{s}\prod_{i\in T} x_i,\] 
 where $0\leq j \leq n-2$, $T\subset \{1,\ldots, n-1\}$ is any subset of size $n-2-j$, and $0\leq s\leq 2j$, form a basis of the finite-dimensional vector space $R_n/J_n$ over $F$. 
\end{enumerate}
\end{proposition}
\begin{proof} We first argue that the elements listed under (i)(a)-(c) are contained in the ideal $J_n$.
For elements in (a), it is enough to note that such a square is the leading term of the element 
\[ x_i^2-x_n^2  = x_i\left(x_i-\prod_{i\neq j}x_j\right) - x_n \left(x_n-\prod_{n\neq j}x_j\right)\in I_n.
\]
The element in (b) is the leading term of one of the original generators of $I_n$. Finally if $0\leq j\leq n-2$ and $T\subset \{1,\ldots, n-1\}$ is any subset  of size $n-2-j$, we claim that the polynomial 
\[ q(T)=x_n^{2j+1}\prod_{i\in T} x_i - \prod_{i\notin T} x_i\]
is in the ideal $I_n$. This would indeed prove that the element $m(T)$ in (c), the leading term of the polynomial $q(T)$, is also in $J_n$. 
For $j=0$, this claim is clearly true, as the expression $q(T)$ is again one of the original generators. We then proceed by induction on $j$: if $T=T'\cup\{l\}$, then by the induction assumption $q(T)\in I_n$, and so 
\[ \prod_{i\notin T} x_i \equiv x_n^{2j+1}\prod_{i\in T} x_i \mod I_n. 
\]
Hence
\begin{eqnarray*} \prod_{i\notin T'} x_i& \equiv & x_l\cdot\prod_{i\notin T} x_i \mod I_n\\ & \equiv & x_l\cdot x_n^{2j+1}\prod_{i\in T} x_i \mod I_n\\ &= &x_l^2\cdot x_n^{2j+1}\prod_{i\in T'} x_i \\ &\equiv& x_n^{2j+3}\prod_{i\in T'} x_i \mod I_n,
\end{eqnarray*}
where in the last congruence, we used the fact that $x_l^2-x_n^2\in I_n$. Hence indeed $q(T')\in I_n$.

We next claim that the elements listed in (ii) form a spanning set of the vector space  $R_n/J_n$. This is just the contrapositive of what we already proved for (i). The $F$-algebra $R_n$ is spanned by all monomials as an $F$-vector space. In the quotient $R_n/J_n$, we do not need the generators which are multiples of any of the monomials listed in (i). A quick reflection shows that we end up with the list of monomials as stated. 

To conclude the proof, it remains to do a count. The number of monomials listed in (ii) is
\[\sum_{j=0}^{n-2} (2j+1) {n-1 \choose n-2-j} = 1 + (n-2)\cdot 2^{n-1}\]
by~\eqref{identity}. This number is precisely the dimension of the finite dimensional $F$-vector space $R_n/J_n$. Hence the monomials $\{m(T,s)\}$ must form a basis of this quotient space as claimed. Correspondingly, the monomials listed in (i) must form a set of generators of the ideal~$J_n$.
\end{proof}

\begin{proof}[Conclusion of the proof of Theorem~\ref{thm2}]  We have already exhibited a monomial basis for the vector space $R_n/J_n$ in Proposition~\ref{prop}(ii). The degree of each basis element is easy to see: 
\[\deg m(T,s) = |T|+s =n-2-j+s.\]
Hence, substituting $k=n-2-j+s$ and $l=n-2-j$ in the second and third lines, we get 
\begin{eqnarray*}
h_n(t) & = & \sum_{j=0}^{n-2} \sum_{s=0}^{2j} {n-1\choose n-2-j} t^{n-2-j+s} \\
& = & \sum_{k=0}^{2n-4}t^k\sum_{j=|n-2-k|}^{n-2} {n-1\choose n-2-j} \\
& = & \sum_{k=0}^{2n-4}t^k\sum_{l=0}^{\min(k,2n-4-k)} {n-1\choose l} \\
& = & \sum_{k=0}^{2n-4} t^k a_{n,k}
\end{eqnarray*}
as claimed. 
\end{proof}


To introduce our final result in this section, note the following. In the list of monomial generators of $J_n$ given in Proposition~\ref{prop}(i), the variable $x_n$ clearly plays a distinguished role, but the other variables arise completely symmetrically. This is unexpected, as the term order we used, indeed any term order, is asymmetric with respect to permutations of the variables. Be as it may, it is clear from our results that the homogeneous ideal $J_n$ is invariant under the symmetric group $S_{n-1}$ permuting the first $(n-1)$ variables. Hence the quotient $F$-algebra $R_n/J_n$ becomes a graded finite-dimensional representation of the group $S_{n-1}$. 
\begin{theorem}\label{thm3} For $n\geq 2$, the graded character of the Artinian $F$-algebra $R_n/J_n$ is given by
\[ \chi_t\left(R_n/J_n \right) = \sum_{k=0}^{2n-4}t^k\sum_{l=0}^{\min(k,2n-4-k)} [\cP_{l}(n-1)] \in \Rep(S_{n-1})[t].
\]
\end{theorem}
\begin{proof} Recall that the monomials $\{m(T,s)\}$ from Proposition~\ref{prop}(ii) form a basis of the quotient. For fixed $s$ and $j$, the set $\{m(T,s)\}$ is an $S_{n-1}$-set
in bijection with $\cP_{n-2-j}(n-1)$. The result then follows by repeating in $\Rep(S_{n-1})[t]$ the calculation used at the end of the proof of Theorem~\ref{thm2}. 
\end{proof}

\section{The homogeneous specialisation}\label{sec4}

Suggested by Joachim Jelisiejew, in this section we study the homogeneous specialisation of~$I_n$, the ideal defined by highest degree terms of all elements of $I_n$. In other words, we consider the initial ideal 
\[ K_n = \mathop{\mathrm{in}_{\mathrm{Gr}}} I_n\lhd R_n\]
defined by the partial order given by polynomial degree. This is a flat specialisation of~$I_n$, still of finite codimension $c_n$. In fact, as $K_n$ is a graded specialisation of $J_n$, the quotient $R_n/K_n$ has the same Hilbert series
\[h_n(t) = \sum_{k=0}^{2n-4} t^k a_{n,k}\]
as $R_n/J_n$, known from Theorem~\ref{thm2}.

\begin{proposition}\label{prop_gen} For $n\geq 3$, the homogeneous specialisation $K_n$ of our ideal $I_n$ can be given explicitly by generators as
\[ K_n = \left( x_1^2-x_n^2,\ldots, x_{n-1}^2-x_n^2, x_2x_3\cdots x_n, x_1x_3\cdots x_n, \ldots, x_1x_2\cdots x_{n-1}\right).
\]
\end{proposition}
\begin{proof} It is clear from formulae in the previous section or direct observation that all the listed generators are contained in $K_n$. Arguments similar to those used to prove Proposition~\ref{prop} then show that the two ideals are equal.  
\end{proof}

The principal difference between the ideals $J_n$ and $K_n$ is the content of the following theorem, whose proof, due to Stavros Argyrios Papadakis, is included in the Appendix. This result explains the palindromic symmetry of the Hilbert series $h_n(t)$.

\begin{theorem}\label{thmG} The Artinian graded $F$-algebra $R_n/K_n$ of codimension $n$ is Gorenstein. 
\end{theorem}

\begin{remark} The proof in the Appendix uses the Kustin--Miller theory of unprojections, which works in more general circumstances also. In our Artinian situation, there are other possible proofs. Miles Reid showed an argument to the authors that proves directly that the socle of $R_n/K_n$ is one-dimensional; it is well known that for graded Artinian rings, this implies the Gorenstein property. Joachim Jelisiejew pointed out that one can also use Macaulay inverse systems~\cite[Sect.21.2]{eis}. In this approach, the ideal $K_n$ corresponds to the inverse system generated by
\[  g_n = \sum_{k\in {\mathcal K}} m_k^2, 
\]
where $\{m_k\colon k\in {\mathcal K}\}$ is the set of all monomials of degree $(n-2)$ in the variables $y_i$ dual to the original variables~$x_i$.
For example, 
\[
g_3 = y_1^2 + y_2^2 + y_3^2
\]
and
\[
g_4 =y_1^4+y_1^2y_2^2+y_2^4+y_1^2y_3^2+y_2^2y_3^2+y_3^4+y_1^2y_4^2+y_2^2y_4^2+y_3^2y_4^2+y_4^4.
\]
Since $K_n$ has a dual socle generated by a single element $g_n$, the quotient $R_n/K_n$ is Gorenstein~\cite[Thm.21.6]{eis}. 

\end{remark}

\begin{remark} It is easy to see that the graded Artinian $F$-algebra $R_n/J_n$, defined by the monomial ideal $J_n$ from the previous section, is not Gorenstein for $n>2$.  
\end{remark}

To conclude the paper, let us note that the ideal $K_n$, defined without any reference to an order between the variables, retains the full $S_n$-symmetry of the inhomogeneous ideal $I_n$. This is also clear from Proposition~\ref{prop_gen}: including all differences of squares presents an $S_n$-equivariant set of generators for $K_n$. We obtain a graded Hilbert series 
\[ \tilde h_n(t) = \chi_t\left(R_n/K_n\right)\in \Rep(S_n)[t],
\]
a further representation-theoretic refinement of the combinatorial quantity $c_n$. Setting $t=1$ in $\tilde h_n(t)$ returns the formula of Theorem~\ref{thm1}, whereas considered non-equivariantly, we get back the formula of Theorem~\ref{thm2}. 

\begin{challenge} Compute the series $\tilde h_n(t)\in \Rep(S_n)[t]$. 
\end{challenge}

\section*{Appendix by Stavros Argyrios Papadakis: the proof of Theorem~\ref{thmG}}

For $n=2$, the statement is obvious. 

For $n \geq 3$, the proof is by induction on $n$.   
(We remark that, alternatively,  for $n=3$, it is easy to check using Macaulay2~\cite{m2} that $R_3/K_3$ is Gorenstein, 
or construct a resolution by hand in $5\times 5$ Pfaffian format.)
We set $R_{n-1}=F[x_1,\ldots, x_{n-1}]$. For $1\leq i \leq n-2$, 
let $h_i =  x_i^2 - x_{n-1}^2$. For $1\leq i \leq n-1$, let $p_i =  x_1\cdots x_{i-1}x_{i+1}\cdots x_{n-1}$.  We have 
\[ K_{n-1}   =  ( h_1, h_2,  ... , h_{n-2} ,   p_1,   ... , p_{n-1})    \lhd R_{n-1} \]
and  define          
\[L_{n-1}  =  ( h_1, h_2,  ... , h_{n-2} ,   x_1x_2\cdots x_{n-1}) \lhd R_{n-1}. \]
Then clearly $L_{n-1}  \subset  K_{n-1}$. An obvious primary decomposition argument gives that $L_{n-1}$ is a 
complete intersection ideal,   and that  both
quotient rings  $ R_{n-1}/L_{n-1}$ and $ R_{n-1}/K_{n-1}$ have the same Krull dimension $0$. 
Moreover,  $R_{n-1}/K_{n-1}$ is Gorenstein by the inductive hypothesis.  Thus 
the pair $L_{n-1} \subset  K_{n-1}$ is a predata for a degenerate Kustin-Miller unprojection in the sense of \cite[Section 4]{APP}. 

We consider the graded $R_{n-1}$-module    $P_{n-1} = (L_{n-1}  :  K_{n-1})/L_{n-1}$.

\renewcommand{\thetheorem}{A.1}
\begin{proposition}\label{prop_firsofappendix}  
The $R_{n-1}$-module    $P_{n-1}$ is cyclic, generated by the element  $ x_{n-1}^2~+~L_{n-1} $.
\end{proposition}

\begin{proof}  The module  $P_{n-1}$  is cyclic by   \cite[Theorem~4.8]{APP}. 
It is graded, where each variable $x_i$ 
has degree $1$.  It is clear that    $x_{n-1}^2  \in (L_{n-1}  :  K_{n-1})$  and that $x_{n-1}^2$ is not
an element of $L_{n-1}$.  Moreover, it is easy to see that if $ 0 \not= u \in  R_{n-1} \setminus  L_{n-1}$ 
is homogeneous of  degree $1$  then u is not an element of $(L_{n-1}  :  K_{n-1})$. The proposition 
follows.  
\end{proof}

Assume that $z$ is a new variable.   Combining Proposition~\ref{prop_firsofappendix}  
with  \cite[Theorem~4.8]{APP} we have that  the Kustin-Miller unprojection ideal of the pair
\[(K_{n-1}, z)/(L_{n-1})  \subset   R_{n-1}[z] / (L_{n-1})\]
is the ideal     
\[ Q_n  =   L_{n-1}  +    ( x_np_1,  ...    , x_np_{n-1},  x_nz - x_{n-1}^2)\lhd R_{n-1}[x_n,z]\cong R_n[z],\]
where we denoted the new variable by $x_n$ instead of the more traditional $T$.
By the general theory of Kustin-Miller unprojection (see, for example, the brief survey  in  \cite[p.~21, Appendix]{APP})
$R_n[z]/Q_n$ is a graded Gorenstein algebra of Krull dimension $1$.

If we substitute $x_n$ for $z$ in $Q_n$, we get the ideal            
\[
    K_n = (x_1^2-x_{n-1}^2,\ldots, x_{n-2}^2-x_{n-1}^2, x_{n}^2-x_{n-1}^2, x_2\cdots x_n, x_1x_3\cdots x_n, \ldots, x_1\cdots x_{n-1})\lhd R_n,
\]
which, as observed above, is an Artinian ideal of $R_n$.  Thus the Krull dimension of the algebra  $R_n[z]/(Q_n , z-x_n)$ is $0$. Since
$R_n[z]/Q_n$ is Gorenstein, hence Cohen-Macaulay,  we get that 
the homogeneous element $z - x_n \in  R_n[z]$ is $R_n[z]/Q_n$-regular.
This implies that  $ R_n[z]/(Q_n, z-x_n)$ is Gorenstein, and since 
\[
       R_n/K_n \cong   R_n[z]/(Q_n, z-x_n)
\]
it follows that    $R_n/K_n$ is Gorenstein, completing the proof of the inductive step.

\end{document}